\documentclass[12pt,a4paper]{amsart}

\usepackage[abbrev]{amsrefs}
\usepackage{amscd}

\theoremstyle{plain}
  \newtheorem{thm}{Theorem}[section]
  \newtheorem{lem}[thm]{Lemma}
  
  \newtheorem{cor}[thm]{Corollary}
  \newtheorem{prop}[thm]{Proposition}

\theoremstyle{definition}
  \newtheorem{defn}[thm]{Definition}

\theoremstyle{remark}

\numberwithin{equation}{section}

\DeclareMathOperator{\vol}{vol}
\DeclareMathOperator{\diam}{diam}

\DeclareMathOperator{\ObsDiam}{ObsDiam}

\DeclareMathOperator{\Sep}{Sep}

\newcommand{\cL}{\mathcal{L}}

\newcommand{\field}[1]{\mathbb{#1}}

\newcommand{\R}{\field{R}}

\begin{document}

\title{Estimate of observable diameter of $l_p$-product spaces}

\thanks{The authors are partially supported by a Grant-in-Aid
for Scientific Research from the Japan Society for the Promotion of Science.
The first author is supported by Research Fellowships of the Japan Society for the Promotion of Science for Young Scientists.}

\begin{abstract}
  We estimate the observable diameter of the $l_p$-product space
  $X^n$ of an mm-space $X$ by using the limit formula in our previous paper
  \cite{OzSy:pyramid}.
  The idea of our proof is based on Gromov's book
  \cite{Gmv:greenbook}*{\S 3.$\frac{1}{2}$}.
  As a corollary we obtain the phase transition property of $\{X^n\}_{n=1}^\infty$
  under a discreteness condition.
\end{abstract}

\author{Ryunosuke Ozawa \and Takashi Shioya}

\address{Mathematical Institute, Tohoku University, Sendai 980-8578,
  JAPAN}

\date{\today}




\maketitle

\section{Introduction}
\label{sec:intro}

In the study of concentration of measure phenomenon discovered
by L\'evy and Milman,
it is an important problem to estimate the observable diameter
(or equivalently the concentration function)
 of a given mm-space,
where an \emph{mm-space} is defined to be a space $X$
equipped with a complete separable metric $d_X$ and
a Borel probability measure $\mu_X$ on $X$.
The purpose of this paper is to estimate
the observable diameter of the \emph{$l_p$-product space}
\[
X_p^n := (X^n,d_{X^n_p},\mu_X^{\otimes n}),
\]
of an mm-space $X$,
where $d_{X^n_p}$ is the \emph{$l_p$-metric}, $1 \le p < +\infty$, defined by
\[
d_{X^n_p}(x,y) :=
  \left( \sum\limits_{i=1}^n d_X(x_i,y_i)^p \right)^{\frac{1}{p}}
\]
for $x = (x_1,x_2,\dots,x_n), y = (y_1,y_2,\dots,y_n) \in X^n$,
and $\mu_X^{\otimes n}$ is the product measure.
We refer to \cites{AmMlm,MlmSch,Ld,Tlg} 
for some significant studies of concentration of measure
for $l_1$-product spaces.

One of our main theorems is stated as follows.

\begin{thm} \label{thm:upper}
  Let $X$ be an mm-space with finite diameter.
  Then, for any $0 < \kappa < 1$ and $1 \le p < +\infty$, we have
  \[
  \ObsDiam(X_p^n;-\kappa) \le C_{\kappa,p} \diam(X) \, n^{\frac{1}{2p}},
  \]
  where
  \[
  C_{\kappa,p} :=
  \begin{cases}
    4\sqrt{2\log\frac{2}{\kappa}} &\text{if $p = 1$}, \\
    4 + 4\sqrt{2\log\frac{2}{\kappa}} &\text{if $0 < p < +\infty$}.
  \end{cases}
  \]
\end{thm}

In the case where $p = 1$, Theorem \ref{thm:upper}
is already known (see \cite{Ld}*{\S 1.6}).
For $p > 1$,
a rough sketch of the proof was given in
\cite{Gmv:greenbook}*{\S 3.$\frac12$.62}.
In this paper, we give a complete proof of the theorem
by using the limit formula for observable diameter
(see Theorem \ref{thm:limf})
established in our previous paper \cite{OzSy:pyramid}.

For lower estimate, we have the following theorems.

\begin{thm} \label{thm:lower-p}
  Let $X$ be a nontrivial mm-space of finite diameter such that $d_X(x,y) \ge \delta$
  for any different $x,y \in X$, where $\delta$ is a positive constant.
  Then, for any $0 < \kappa < 1$ and $1 < p < +\infty$, we have
  \[
  \liminf_{n\to\infty} \frac{\ObsDiam(X_p^n;-\kappa)}{n^{\frac{1}{2p}}}
  \ge \left( \frac{2\,\Phi^{-1}(\kappa) \sqrt{V_p(X)}}{p} \right)^{\frac{1}{p}}
  \left( \frac{\delta}{\diam(X)} \right)^{1-\frac{1}{p}},
  \]
  where
  \[
  \Phi(r) := \int_r^\infty \frac{1}{\sqrt{2\pi}} e^{-\frac{t^2}{2}}\;dt, \qquad
  V_p(X) := \sup_{x \in X} V(d_X(x,\cdot)^p),
  \]
  and $V(f)$ is the variance of a function $f : X \to \R$, i.e.,
  $V(f) := \int_X \left(f - \int_X f d\mu_X \right)^2\;d\mu_X$.
\end{thm}

If $p = 1$, then the $\delta$-discreteness assumption is not necessary.

\begin{thm} \label{thm:lower-1}
 Let $X$ be a nontrivial mm-space of finite diameter.
 Then, for any $0 < \kappa < 1$, we have
 \[
 \liminf_{n\to\infty} \frac{\ObsDiam(X_1^n;-\kappa)}{\sqrt{n}}
 \ge 2\,\Phi^{-1}(\kappa) \sqrt{V_1(X)}.
 \]
\end{thm}

Gromov \cite{Gmv:greenbook}*{\S 3.$\frac{1}{2}$.42}
proved Theorem \ref{thm:lower-p} for $X = \{0,1\}$ and $p = 2$.

For Theorem \ref{thm:lower-p}, the $\delta$-discreteness assumption for $X$
is necessary.  In fact, we see from a result of \cite{GmvMlm} that
$\ObsDiam(X^n_2;-\kappa)$ is bounded from above
if $X$ is a compact and connected Riemannian manifold
(see Proposition \ref{prop:prod-mfd}).

In the case where $p = +\infty$,
we have the trivial uniform estimates
\[
0 < \ObsDiam(X;-\kappa) \le \ObsDiam(X_\infty^n;-\kappa) \le \diam(X)
\]
for $X$ nontrivial, 
where $d_{X^n_\infty}(x,y) := \max\limits_{i=1}^n d_X(x_i,y_i)$.

Combining Theorems \ref{thm:upper}, \ref{thm:lower-p}, and \ref{thm:lower-1}
tells us that the order of the observable diameter
$\ObsDiam(X^n_p;-\kappa)$ as $n \to \infty$
is $n^{\frac{1}{2p}}$, which is independent of $\kappa$.
This together with Theorem \cite{OzSy:pyramid}*{Theorem 1.2}
implies the following corollary.

\begin{cor}
  Let $X$ be a nontrivial mm-space of finite diameter and let $1 \le p \le +\infty$.
  We assume that $\inf_{x,y \in X,\ x \neq y} d_X(x,y) > 0$ if $1 < p < +\infty$.
  Then, $\{X^n_p\}_{n=1}^\infty$ has the phase transition property
  with critical scale order $\sim n^{-\frac{1}{2p}}$, i.e.,
  \begin{enumerate}
  \item if $n^{\frac{1}{2p}} t_n \to 0$ as $n\to\infty$,
    then $\{t_n X^n_p\}$ is a L\'evy family;
  \item if $n^{\frac{1}{2p}} t_n \to +\infty$ as $n\to\infty$,
    then $\{t_n X^n_p\}$ $\infty$-dissipates,
  \end{enumerate}
  where $t_n X^n_p$ is the scale change of $X^n_p$ of factor $t_n$.
\end{cor}

A L\'evy family is a sequence of mm-spaces with observable diameter
tending to zero.
An $\infty$-dissipating sequence is an opposite notion to a L\'evy family
and means that
the mm-spaces disperse into many small pieces far apart each other
(see \cites{Gmv:green book,OzSy:pyramid,Sy:book} for the details).
The corollary claims the asymptotic uniformity of the $l_p$-product spaces
$X^n_p$ in a sense.

\section{Preliminaries}

In this section, we state some definitions and
facts needed in this paper.
We refer to \cites{Gmv:greenbook, Sy:book} for more details.

\begin{defn}[Lipschitz order] \label{defn:dom}
  Let $X$ and $Y$ be two mm-spaces.
  We say that $X$ (\emph{Lipschitz}) \emph{dominates} $Y$
  and write $Y \prec X$ if
  there exists a $1$-Lipschitz map $f : X \to Y$ with
  $f_*\mu_X = \mu_Y$,
  where $f_*\mu_X$ is the push-forward measure of $\mu_X$ by $f$.
  We call the relation $\prec$ the \emph{Lipschitz order}.
\end{defn}

\begin{defn}[Partial and observable diameter] \label{defn:ObsDiam}
  Let $X$ be an mm-space.
  For a real number $\alpha$, we define
  the \emph{partial diameter
    $\diam(X;\alpha) = \diam(\mu_X;\alpha)$ of $X$}
  to be the infimum of $\diam(A)$,
  where $A \subset X$ runs over all Borel subsets
  with $\mu_X(A) \ge \alpha$ and $\diam(A)$ denotes the diameter of $A$.
  For a real number $\kappa > 0$, we define
  the \emph{observable diameter of $X$} to be
  \begin{align*}
    \ObsDiam(X;-\kappa) &:= \sup\{\;\diam(f_*\mu_X;1-\kappa) \mid\\
    &\qquad\qquad\text{$f : X \to \R$ is $1$-Lipschitz continuous}\;\}.
  \end{align*}
\end{defn}

Clearly, $\ObsDiam(X;-\kappa)$ is monotone nonincreasing in $\kappa > 0$.
Note that $\ObsDiam(X;-\kappa) = \diam(X;1-\kappa) = 0$ for $\kappa \ge 1$.


For an mm-space $X$ and a real number $t > 0$, we define
$tX$ to be the mm-space $X$ with the scaled metric $d_{tX} := t d_X$.

\begin{prop}
  Let $X$ be an mm-space.
  Then we have
  \[
  \ObsDiam(tX;-\kappa) = t \ObsDiam(X;-\kappa)
  \]
  for any $t,\kappa > 0$.
\end{prop}

\begin{prop} \label{prop:ObsDiam-dom}
  If $X \prec Y$, then
  \[
  \ObsDiam(X;-\kappa) \le \ObsDiam(Y;-\kappa)
  \]
  for any $\kappa > 0$.
\end{prop}

\begin{defn}
  The \emph{concentration function} $\alpha_X(r)$, $r > 0$,
  of an mm-space $X$ is defined by
  \[
  \alpha_X(r) := \sup\{\; 1-\mu_X(A_r) \mid
  A \subset X : \text{Borel}, \ \mu_X(A) \ge 1/2\;\},
  \]
  where $A_r$ denotes the open $r$-neighborhood of $A$.
\end{defn}

\begin{prop}[see \cite{Ld}*{Proposition 1.12}, \cite{Sy:book}*{Remark 2.28}]
   \label{prop:alpha-ObsDiam}
   We have
  \begin{align}
    \tag{1} &\ObsDiam(X;-\kappa)
    \le 2\,\inf\{\; r > 0 \mid \alpha_X(r) \le \kappa/2\;\},\\
    \tag{2} &\alpha_X(r)
    \le \sup\{\;\kappa > 0 \mid \ObsDiam(X;-\kappa) \ge r\;\}.
  \end{align}
\end{prop}

\begin{defn}[Separation distance] \label{defn:Sep}
  Let $X$ be an mm-space.
  For any real numbers $\kappa_0,\kappa_1,\cdots,\kappa_N > 0$
  with $N\geq 1$,
  we define the \emph{separation distance}
  \[
  \Sep(X;\kappa_0,\kappa_1, \cdots, \kappa_N)
  \]
  of $X$ as the supremum of $\min_{i\neq j} d_X(A_i,A_j)$
  over all sequences of $N+1$ Borel subsets $A_0,A_2, \cdots, A_N \subset X$
  satisfying that $\mu_X(A_i) \geq \kappa_i$ for all $i=0,1,\cdots,N$,
  where $d_X(A_i,A_j) := \inf_{x\in A_i,y\in A_j} d_X(x,y)$.
  If there exists no sequence $A_0,\dots,A_N \subset X$
  with $\mu_X(A_i) \ge \kappa_i$, $i=0,1,\cdots,N$, then
  we define
  \[
  \Sep(X;\kappa_0,\kappa_1, \cdots, \kappa_N) := 0.
  \]
\end{defn}

\begin{prop} \label{prop:ObsDiam-Sep}
  For any mm-space $X$ and any real numbers $\kappa$ and $\kappa'$
  with $\kappa > \kappa' > 0$, we have
  \[
  \ObsDiam(X;-2\kappa) \le \Sep(X;\kappa,\kappa)
  \le \ObsDiam(X;-\kappa').
  \]
\end{prop}

\begin{defn}[Parameter]
  Let $I := [\,0,1\,)$ and let $X$ be an mm-space.
  A map $\varphi : I \to X$ is called a \emph{parameter of $X$}
  if $\varphi$ is a Borel measurable map such that
  \[
  \varphi_*\cL^1 = \mu_X,
  \]
  where $\cL^1$ denotes the one-dimensional Lebesgue measure on $I$.
\end{defn}

Any mm-space has a parameter.

\begin{defn}[Box distance]
  We define the \emph{box distance $\square(X,Y)$ between
    two mm-spaces $X$ and $Y$} to be
  the infimum of $\varepsilon \ge 0$
  satisfying that there exist parameters
  $\varphi : I \to X$, $\psi : I \to Y$, and
  a Borel subset $I_0 \subset I$ such that
  \begin{align}
    & |\,\varphi^*d_X(s,t)-\psi^*d_Y(s,t)\,| \le \varepsilon
    \quad\text{for any $s,t \in I_0$};\tag{1}\\
    & \cL^1(I_0) \ge 1-\varepsilon,\tag{2}
  \end{align}
  where
  $\varphi^*d_X(s,t) := d_X(\varphi(s),\varphi(t))$ for $s,t \in I$.
\end{defn}

The box distance function $\square$ is a complete separable metric on
the set of isomorphism classes of mm-spaces.

The following is a corollary to \cite{OzSy:pyramid}*{Theorem 1.1}.

\begin{thm}[Limit formula for observable diameter; \cite{OzSy:pyramid}]
  \label{thm:limf}
  Let $X$ and $X_n$, $n=1,2,\dots$, be mm-spaces.
  If $X_n$ $\square$-converges to $X$ as $n\to\infty$, then
  \begin{align*}
    \ObsDiam(X;-\kappa)
    &= \lim_{\varepsilon\to 0+} \liminf_{n\to\infty}
    \ObsDiam(X_n;-(\kappa+\varepsilon)) \\
    &= \lim_{\varepsilon\to 0+} \limsup_{n\to\infty}
    \ObsDiam(X_n;-(\kappa+\varepsilon))
  \end{align*}
  for any $\kappa > 0$.
\end{thm}

\section{Proof}

We need some lemmas for the proof of Theorem \ref{thm:upper}.

\begin{defn}
  Let $k$ be a natural number.
  A \emph{$k$-regular} mm-space is defined to be
  a $k$-point mm-space $X$ satisfying that
  $d_X(x,y) = 1$ for any different two points $x,y \in X$
  and $\mu_X = \frac{1}{k} \sum_{x \in X} \delta_x$,
  where $\delta_x$ denotes the Dirac measure at $x$.
\end{defn}

\begin{lem} \label{lem:k-reg}
  Let $X$ be a $k$-regular mm-space.
  Then, for any $0 < \kappa < 1$ and $1 \le p < +\infty$,
  we have
  \[
  \ObsDiam(X_p^n;-\kappa)
  \le n^{\frac{1}{2p}-\frac{1}{2}} \ObsDiam(X_1^n;-\kappa) + 4n^{\frac{1}{2p}}.
  \]
\end{lem}

\begin{proof}
  Let $f : X_p^n \to \R$ be an arbitrary $1$-Lipschitz function.
  For $\delta := n^{\frac{1}{2p}}$, we
  take a maximal $\delta$-separated net $Y \subset X_p^n$, i.e.,
  $Y$ is a maximal set of $X$ such that any different two points in $Y$
  has $l_p$-distance at least $\delta$.
  We prove that $f$ is $\delta^{1-p}$-Lipschitz with respect to $d_{X^n_1}$.
  In fact, since $X$ is $k$-regular, for any different two points $x,x' \in Y$ we have
  \[
  |f(x) - f(x')| \le d_{X^n_p}(x,x') = d_{X^n_1}(x,x')^{\frac{1}{p}}.
  \]
  So, it suffices to prove that $d_{X^n_1}(x,x')^{\frac{1}{p}} \le \delta^{1-p} d_{X^n_1}(x,x')$,
  which follows from
  \[
  \delta^{1-p} d_{X^n_1}(x,x')^{1-\frac{1}{p}}
  = \delta^{1-p} d_{X^n_p}(x,x')^{p-1} \ge \delta^{1-p} \delta^{p-1} = 1,
  \]
  since $d_{X^n_p}(x,x') \ge \delta$.
  
  There is a $\delta^{1-p}$-Lipschitz extension $f' : X_1^n \to \R$ of $f|_Y$.
  We have
  \begin{align*}
    \delta^{p-1} \diam(f'_*\mu_{X_1^n};1-\kappa)
    &= \diam((\delta^{p-1}f')_*\mu_{X_1^n};1-\kappa) \\
    &\le \ObsDiam(X_1^n;-\kappa)
  \end{align*}
  and hence
  \[
  \diam(f'_*\mu_{X_1^n};1-\kappa) \le
  n^{\frac{1}{2p}-\frac{1}{2}} \ObsDiam(X_1^n;-\kappa).
  \]
  
  Let us estimate the distance between $f$ and $f'$.
  For any point $x \in X_1^n$ there is a point $x' \in Y$ such that
  $d_{X^n_p}(x,x') \le \delta$.
  Then, since $d_{X^n_1}(x,x') = d_{X^n_p}(x,x')^p \le \delta^p$, we see
  \begin{align*}
  |f(x)-f'(x)| &\le |f(x)-f(x')| + |f(x')-f'(x')| + |f'(x')-f'(x)| \\
  &\le d_{X^n_p}(x,x') + \delta^{1-p} d_{X^n_1}(x,x') \\
  &\le \delta + \delta^{1-p} \delta^p = 2\delta.
  \end{align*}
  Therefore,
  \begin{align*}
    \diam(f_*\mu_{X_p^n};1-\kappa) &\le \diam(f'_*\mu_{X_1^n};1-\kappa) + 4\delta \\
    &\le  n^{\frac{1}{2p}-\frac{1}{2}} \ObsDiam(X_1^n;-\kappa) + 4n^{\frac{1}{2p}}.
  \end{align*}
  By the arbitrariness of $f$, this completes the proof.
\end{proof}

\begin{lem} \label{lem:box-np}
  Let $X$ and $Y$ be two mm-spaces.  For any natural number $n$
  and real number $p$ with $1 \le p \le +\infty$, we have
  \[
  \square(X_p^n,Y_p^n) \le n\cdot n! \, \square(X,Y).
  \]
\end{lem}

\begin{proof}
  We prove the lemma only in the case of $1 \le p < +\infty$.
  The proof for $p = +\infty$ is similar.
  
  Note that $\square \le 1$ and that the lemma is trivial if $\square(X,Y) = 1$.
  Assume $\square(X,Y) < \varepsilon \le 1$.
  It suffices to prove that $\square(X_p^n,Y_p^n) \le n\cdot n!\,\varepsilon$.
  There are parameters $\varphi : I \to X$, $\psi : I \to Y$,
  and a subset $I_0 \subset X$ such that $\cL^1(I_0) > 1-\varepsilon$ and
  \[
  |\;\varphi^*d_X(s,t) - \psi^*d_Y(s,t)\;| < \varepsilon
  \]
  for any $s,t \in I_0$.
  Define functions $\varphi_n : I^n \to X_p^n$ and $\psi_n : I^n \to Y_p^n$ by
  \begin{align*}
  \varphi_n(s_1,\dots,s_n) &:= (\varphi(s_1),\dots,\varphi(s_n)), \\
  \psi_n(s_1,\dots,s_n) &:= (\psi(s_1),\dots,\psi(s_n))
  \end{align*}
  for $(s_1,\dots,s_n) \in I^n$.
  For any $s = (s_1,\dots,s_n), t = (t_1,\dots,t_n) \in I_0^n$, we have
  \begin{align*}
    & |\; \varphi_n^*d_{X_p^n}(s,t) - \psi_n^*d_{Y_p^n}(s,t) \;| \\
    &= \left| \left( \sum_{i=1}^n d_X(\varphi(s_i),\varphi(t_i))^p \right)^{\frac{1}{p}}
    - \left( \sum_{i=1}^n d_Y(\psi(s_i),\psi(t_i))^p \right)^{\frac{1}{p}} \right| \\
    &\le \left( \sum_{i=1}^n |\; d_X(\varphi(s_i),\varphi(t_i) - d_Y(\psi(s_i),\psi(t_i))\;|^p
    \right)^{\frac{1}{p}} \\
    &< n^{\frac{1}{p}} \varepsilon
  \end{align*}
  and
  \[
    \cL^n(I_0^n) \ge \cL^1(I_0)^n > (1-\varepsilon)^n
    \ge 1 - n\cdot n! \, \varepsilon.
  \]
  Since there is a Borel isomorphism from $I$ to $I^n$ that pushes
  $\cL^1$ to $\cL^n$ (see \cite{Kch}*{(17.41)}), the proof is completed.
\end{proof}

\begin{proof}[Proof of Theorem \ref{thm:upper}]
  It suffices to prove the theorem only in the case where $\diam(X) = 1$,
  because of
  \[
  \ObsDiam(X_p^n;-\kappa) = \diam(X) \ObsDiam((\diam(X)^{-1}X)_p^n;-\kappa)
  \]
  for $\diam(X) > 0$.
  We assume $\diam(X) = 1$.

  If $p = 1$, then \cite{Ld}*{Corollary 1.7} implies
  \[
  \alpha_{X^n_1}(r) \le \exp\left(-\frac{r^2}{8\diam(X)^2 n}\right),
  \]
  which together with Proposition \ref{prop:alpha-ObsDiam}(1) proves the theorem.

  Assume $p > 1$.
  By \cite{Sy:book}*{\S 4}, there is a sequence of finite mm-spaces $Y_k$,
  $i=1,2,\dots$, with diameter $\le 1$ that
  $\square$-converges to $X$.
  In addition, we may assume that the $\mu_{Y_k}$-measure of each point in $Y_k$ 
  is an integral multiplication of $1/k$.
  In particular, $Y_k$ consists of at most $k$ points.
  Denoting by $Z_k$ a $k$-regular mm-space,
  we see that $Y_k \prec Z_k$ for each $k$.
  By $(Y_k)_p^n \prec (Z_k)_p^n$, Lemma \ref{lem:k-reg},
  and by the theorem for $p = 1$, we have
  \begin{align*}
  \ObsDiam((Y_k)_p^n;-\kappa) &\le \ObsDiam((Z_k)_p^n;-\kappa)\\
  &\le n^{\frac{1}{2p}-\frac{1}{2}} \ObsDiam((Z_k)_1^n;-\kappa) + 4n^{\frac{1}{2p}}\\
  &\le C_{\kappa,p} \, n^{\frac{1}{2p}}.
  \end{align*}
  Lemma \ref{lem:box-np} implies that $(Y_k)_p^n$ $\square$-converges
  to $X_p^n$ as $k \to \infty$.
  The limit formula (Theorem \ref{thm:limf}) proves
  \[
  \ObsDiam(X_p^n;-\kappa) = \lim_{\varepsilon\to 0+} \liminf_{k\to\infty}
  \ObsDiam((Y_k)_p^n;-(\kappa+\varepsilon))
  \le C_{\kappa,p} \, n^{\frac{1}{2p}}.
  \]
  This completes the proof of the theorem.
\end{proof}

\begin{proof}[Proof of Theorems \ref{thm:lower-p} and \ref{thm:lower-1}]
  Take an arbitrary point $x_0 \in X$ and fix it.
  We define
  \[
  f_n(x) := \sum_{i=1}^n d_X(x_0,x_i)^p
  \]
  for $x = (x_1,x_2,\dots,x_n) \in X^n$.
  Note that $f_n$ is the sum of a sequence of independent and identically distributed
  random variables.
  The mean of $f_n$ is
  \[
  \int_{X^n} f_n\;d\mu_{X}^{\otimes n} = n E_p(x_0),
  \]
  where $E_p(x_0) := \int_X d_X(x_0,\cdot)^p \; d\mu_X$.
  The variance of $f_n$ is
  \[
  \int_{X^n} (f_n - n E_p(x_0))^2\;d\mu_X^{\otimes n} = n V_p(x_0),
  \]
  where $V_p(x_0) := \int_X (d_X(x_0,\cdot)^p - E_p(x_0))^2 \; d\mu_X$.
  The central limit theorem tells us that
  \begin{equation}
    \tag{CLT} \left( \frac{f_n - n E_p(x_0)}{\sqrt{n V_p(x_0)}} \right)_* \mu_X^{\otimes n}
    \to \gamma^1 \quad \text{weakly as $n\to\infty$}.
  \end{equation}
  Take a real number $\kappa$ with $0 < \kappa < 1/2$ and let
  \begin{align*}
  \alpha_n &:= \inf \{\; x \in \R \mid (f_n)_*\mu_X(\,-\infty,x\,] \ge \kappa \;\}, \\
  \beta_n &:= \sup \{\; x \in \R \mid (f_n)_*\mu_X[\,x,+\infty\,) \ge \kappa \;\}, \\
  c_n &:= \frac{\beta_n-\alpha_n}{\sqrt{n V_p(x_0)}}.
  \end{align*}
  Then, (CLT) proves that
  $\lim_{n\to\infty} c_n = 2\Phi^{-1}(\kappa)$.
  Letting
  \[
  A_n := f_n^{-1}(\,-\infty,\alpha_n\,] \quad\text{and}\quad
  B_n := f_n^{-1}[\,\beta_n,+\infty\,),
  \]
  we have
  \[
  \mu_X^{\otimes n}(A_n),\mu_X^{\otimes n}(B_n) \ge \kappa.
  \]
  Let us estimate $d_{X_p^n}(A_n,B_n)$ from below.
  For any points $x \in A_n$ and $y \in B_n$,
  we see that
  \begin{align*}
  c_n\sqrt{nV_p(x_0)} &= \beta_n - \alpha_n
  \le |\;f_n(x) - f_n(y)\;| \\
  &\le \sum_{i=1}^n |\; d_X(x_0,x_i)^p - d_X(x_0,y_i)^p \;| \\
  \intertext{and by $|a^p - b^p| \le p \cdot\max\{a,b\}^{p-1} |a-b|$ for $a,b \ge 0$,}
  &\le p \diam(X)^{p-1} \sum_{i=1}^n |\; d_X(x_0,x_i) - d_X(x_0,y_i) \;| \\
  &\le p \diam(X)^{p-1} \sum_{i=1}^n d_X(x_i,y_i) \\
  &= p \, \delta \diam(X)^{p-1} \sum_{i=1}^n \frac{d_X(x_i,y_i)}{\delta} \\
  \intertext{and by the assumption $\inf_{x \neq y} d_X(x,y) \ge \delta$,}
  &\le p \, \delta \diam(X)^{p-1} \sum_{i=1}^n \frac{d_X(x_i,y_i)^p}{\delta^p} \\
  &= p \, \delta^{1-p} \diam(X)^{p-1} d_{X^n_p}(x,y)^p,
  \end{align*}
  which implies
  \[
  d_{X^n_p}(x,y) \ge \left( \frac{c_n\sqrt{V_p(x_0)}}{p} \right)^{\frac{1}{p}}
  \left( \frac{\delta}{\diam(X)} \right)^{1-\frac{1}{p}} n^{\frac{1}{2p}}.
  \]
  In the case where $p = 1$, we do not need the assumption
  $\inf_{x \neq y} d_X(x,y) \ge \delta$ and obtain
  \[
  d_{X^n_1}(x,y) \ge c_n\, \sqrt{nV_p(x_0)}.
  \]
  Thus, if $0 < \kappa' < \kappa$, then
  \begin{align*}
    \ObsDiam(X_p^n;-\kappa') &\ge \Sep(X_p^n;\kappa,\kappa) \\
    &\ge \left( \frac{c_n\sqrt{V_p(x_0)}}{p} \right)^{\frac{1}{p}}
    \left( \frac{\delta}{\diam(X)} \right)^{1-\frac{1}{p}} n^{\frac{1}{2p}}.
  \end{align*}
  This completes the proof.
\end{proof}

\begin{prop}[\cite{GmvMlm}, \cite{Gmv:greenbook}*{\S 3.$\frac{1}{2}$.42}] \label{prop:prod-mfd}
  Let $X$ be a compact and connected Riemannian manifold with
  dimension $\ge 1$, and let $0 < \kappa < 1$.
  We assume the measure $\mu_X$ to be the normalized volume measure on $X$,
  i.e., $\mu_X = \vol/\vol(X)$, where $\vol$ is the volume measure on $X$.
  Then, the sequence $\{\ObsDiam(X_2^n;-\kappa)\}_{n=1}^\infty$
  is bounded from above and bounded away from zero.
\end{prop}

\begin{proof}
  Since a natural projection $X_2^n \to X$ induces $X \prec X_2^n$,
  we have the lower estimate
  \[
  0 < \ObsDiam(X;-\kappa) \le \ObsDiam(X_2^n;-\kappa).
  \]

  Denote by $\lambda_1(X)$ the first non-zero eigenvalue of the Laplacian on $X$.
  By $\lambda_1(X_2^n) = \lambda_1(X)$, we have the upper estimate
  \[
  \ObsDiam(X_2^n;-\kappa) \le \frac{2\sqrt{2}}{\sqrt{\lambda_1(X)\kappa}}
  \]
  (see \cite{Sy:book}*{Corollary 2.39}).
  This proves the proposition.
\end{proof}

Combining Proposition \ref{prop:prod-mfd} with
Theorem \cite{OzSy:pyramid}*{Theorem 1.2} implies

\begin{cor}
  If $X$ is a compact and connected Riemannian manifold with
  dimension $\ge 1$, then
  $\{X_2^n\}_{n=1}^\infty$ has the phase transition property
  with critical scale order $\sim 1$.
\end{cor}

\end{document}